\documentclass[12pt]{amsart}
\usepackage{amsmath}
\usepackage{amssymb}
\usepackage{amsthm}
\usepackage[doublespacing]{setspace}
\usepackage{fullpage}

\newcommand{\textd}{\text{d}}

\newcommand{\bfone}{\mathbf{1}}

\newcommand{\bbn}{\mathbb{N}}

\newcommand{\ep}{\epsilon}

\newcommand{\bbp}{\mathbb{P}}
\newcommand{\cals}{\mathcal{S}}

\newcommand{\length}{\text{length}}
\swapnumbers
\theoremstyle{plain}
\newtheorem{theorem}{Theorem}[section]
\newtheorem{proposition}[theorem]{Proposition}
\newtheorem{lemma}[theorem]{Lemma}

\newtheorem{corollary}[theorem]{Corollary}

\theoremstyle{definition}
\newtheorem{definition}[theorem]{Definition}
\newtheorem{example}[theorem]{Example}

\newtheorem*{question*}{Question}
\theoremstyle{remark}

\newtheorem*{claim*}{Claim}

\newcommand{\partition}{P}
\newcommand{\textl}{\text{L}}
\newcommand{\texts}{\text{S}}
\newcommand{\val}{\text{val}}
\newcommand{\Bpsi}{B}
\newcommand{\bpsi}{b}
\newcommand{\zsigma}{z}
\newtheorem*{minimax-theorem}{The Minimax Theorem}
\newtheorem*{martins-theorem}{Martin's Theorem}

\begin{document}

\title[Eventual perfect monitoring]{The determinacy of infinite games with eventual perfect monitoring}
\author{Eran Shmaya}\email{e-shmaya@kellogg.northwestern.edu}\address{Kellogg School of Management\\ Northwestern University}
\date{\today}
\thanks{Thanks to to the anonymous referee for helpful suggestions and comments, and also to Chris Chambers, Tzachi Gilboa, John Levy, Ehud Lehrer, Wojciech Olszewski, Phil Reny, Eilon Solan, Bill Sudderth and Rakesh Vohra.} 
\begin{abstract}An infinite two-player zero-sum game with a Borel winning set, in which the opponent's actions are monitored eventually but not necessarily immediately after they are played, is determined. The proof relies on a representation of the game as a stochastic game with perfect information, in which Chance operates as a delegate for the players and performs the randomizations for them, and on Martin's Theorem about determinacy of such games.\end{abstract}
\maketitle
\section{Setup}\label{sec-setup}
Consider an infinite two-player zero-sum game that is given by a triple $\left(A,\left(\partition_n\right)_{n\in\bbn},W\right)$ where $A$ is a finite set of \emph{actions}, $\partition_n$ is a partition of $A^n$ for every $n\in\bbn$, the \emph{information partition of stage $n$}, and $W\subseteq A^\bbn$ is a Borel set, the \emph{winning set} of player 1. The game is played in stages: player 1 chooses an action $a_0\in A$; then player 2 chooses an action $a_1\in A$; then player 1 chooses an action $a_2\in A$, and so on, ad infinitum. Before choosing $a_n$, the player who plays at stage $n$ receives some information about the actions of previous stages: Let $h=(a_0,a_1,\dots,a_{n-1})$ be the \emph{finite history} that consists of the actions played before stage $n$; then before choosing $a_n$, the player who plays at stage $n$ observes the atom of $\partition_n$ that contains $h$.  
Player 1 wins the game if the \emph{infinite history} $(a_0,a_1,\dots)$ is in $W$. When the action set and information partitions are fixed, I denote the game by $\Gamma(W)$.

A \emph{behavioral strategy} $x=(x_n)_{n\in\bbn}$ of player 1 is a sequence $\{x_n: \partition_n\rightarrow \Delta(A)\}_{n=0,2,4,\dots}$ of functions: At stage $n$, after observing the finite history $h=(a_0,a_1,\dots,a_{n-1})$, player 1 randomizes his action according to $x_n(\pi_n(h))$, where $\pi_n(h)$ is the atom of $\partition_n$ that contains $h$. Abusing notations, I sometimes write $x_n(h)$ instead of $x_n(\pi_n(h))$. Behavioral strategies $y$ of player 2 are defined analogously. 

Every pair $x,y$ of strategies induces a probability distribution $\mu_{x,y}$ over the set $A^\bbn$ of infinite histories or \emph{plays}: $\mu_{x,y}$ is the joint distribution of a sequence $\alpha_0,\alpha_1,\dots\dots$ of $A$-valued random variables such that
\begin{equation}\label{mu-x-y-def}
\bbp\left(\alpha_n=a\left|\alpha_0,\dots,\alpha_{n-1}\right.\right)=
 \begin{cases}x_n(\alpha_0,\dots,\alpha_{n-1})[a],&\text{ if }n\text{ is even},\\
 y_n(\alpha_0,\dots,\alpha_{n-1})[a],&\text{ if }n\text{ is odd}.\end{cases}
\end{equation}
I call such a sequence of random variables an $(x,y)$-\emph{random play}. If the players play according to the strategy profile $(x,y)$, then the expected payoff for player 1 is given by
\begin{equation}\label{expected-via-alpha}\mu_{x,y}(W) = 
\bbp\left((\alpha_0,\alpha_1,\dots)\in W\right),
\end{equation}
where $\alpha_0,\alpha_1,\dots$ is an $(x,y)$-random play
.

The \emph{lower value} $\underline\val~\Gamma(W)$ and \emph{upper value} $\overline\val~\Gamma(W)$ of the game $\Gamma(W)$ are defined by:
\[\underline\val~\Gamma(W)=\sup_x\inf_y \mu_{x,y}(W),\quad\text{and }
\overline\val~\Gamma(W)=\inf_y\sup_x  \mu_{x,y}(W),\] where the suprema are taken over all strategies $x$ of player 1 and the infima over all strategies $y$ of player 2. The game is \emph{determined} if the lower and upper values are equal, $\underline\val~\Gamma(W)=\overline\val~\Gamma(W)$, in which case their common value is called the \emph{value} of the game.  
For $\ep\geq 0$, a strategy $x$ of player 1 is \emph{$\ep$-optimal} if $\mu_{x,y}(W)\geq \underline\val~\Gamma(W)-\ep$ for every strategy $y$ of player 2. We also say that player 1 can \emph{guarantee} payoff of at least $\underline\val~\Gamma(W)-\ep$
 by playing such a strategy $x$. $\ep$-optimal strategies of player 2 are defined analogously. 

Let $\sim_n$ be the equivalence relation over infinite histories such that $u\sim_n u'$ whenever $u|_n$ and $u'|_n$ belong to the same atom of $\partition_n$, where $u|_n$ and $u'|_n$ are the initial segments of $u$ and $u'$ of length $n$.
The interpretation is that if $u,u'\in A^\bbn$ and $u\sim_n u'$, then at stage $n$ the player cannot distinguish between $u$ and $u'$. Say that \emph{at stage $n$ the player observes the action of stage $m$} if, for every pair  of infinite histories $u=(a_0,a_1,\dots)$ and $u'=(a_0',a_1',\dots)$, $u\sim_nu'$ implies $a_m=a_m'$.
\begin{definition}The information partitions $(\partition_n)_{n\geq 0}$ satisfy \emph{perfect recall} if the following conditions are satisfied:
\begin{enumerate}\item Players know their own actions: at stage $n$ the player observes the action of stage $n-2$. 
\item Players do not forget information: if $u,u'\in A^\bbn$ and $u\sim_{n+2}u'$ then $u\sim_{n}u'$.
\end{enumerate}\end{definition}\medskip
The setup of infinite games with perfect recall is general enough to subsume two cases which have been extensively studied:

\noindent\textbf{Gale-Stewart games.} If, at every stage $n$, players observe previous actions of their opponents, then the game is called a \emph{Gale-Stewart game} or a \emph{game with perfect information}. Gale and Stewart~\cite{gale-stewart-53} proved that such games are determined if the winning set $W$ is closed (or open) and asked whether the result can be extended to more complex winning sets. In a seminal paper, Martin~\cite{martin-75} proved that the game is determined for every Borel winning set $W$. Gale-Stewart games admit pure $0$-optimal strategies, and the value is $0$ or $1$. Moreover, Gale-Stewart games with an infinite action set $A$ and Borel winning set are also determined.

\noindent\textbf{Blackwell games.} Assume that at even stages $n=2k$, player 1 observes the actions of stages $0,1,\dots,2k-1$, and at odd stages $n=2k+1$, player 2 observes the actions of stages $0,1,\dots,2k-1$ (his own actions and all the previous actions of his opponent except for the last one), and that the information partitions are the roughest partitions that satisfy these conditions.
This essentially means that the players choose actions simultaneously and independently at stages $2k$ and $2k+1$, and then both actions are made public. Such games are called \emph{Blackwell games}. Blackwell~\cite{blackwell-69,blackwell-89} proved the determinacy of Blackwell games (which he called ``infinite games with imperfect information'') with a $G_\delta$ winning set, and conjectured that every Blackwell game with a Borel winning set is determined. Vervoort~\cite{vervoort-96} advanced higher in the Borel hierarchy, proving determinacy of games with $G_{\delta\sigma}$ winning sets. 
Blackwell's conjecture was proved by Martin in 1998~\cite{martin-98}.

Gale-Stewart games and Blackwell games differ in the timing of \emph{monitoring} -- the observation of the opponent's actions: whereas in Gale-Stewart games monitoring is immediate, in Blackwell games player 2's monitoring is delayed by one stage. Both setups satisfy a property that I call \emph{eventual perfect monitoring}, which means that the entire history of the game is known to every player at infinity.
One example of eventual perfect monitoring, of which Blackwell games are a special case, is \emph{delayed monitoring}, introduced by Scarf and Shapley~\cite{scarf-shapley-57}, when the action of stage $m$ is monitored after some lag $d_m$. But the setup of games with eventual perfect monitoring is more general than the setup of games with delayed monitoring. First, the former setup allows the length of the lag to depend on the history of the games. Second, it allows the information to be revealed in pieces; for example, a player can observe some function of the previous actions of his opponent before he observes the actions themselves. 
\begin{definition}
The information partitions $\{P_n\}_{n\in \bbn}$ satisfy \emph{eventual perfect monitoring} if for every $u,u'\in A^\bbn$ such that $u\neq u'$, there exists an $n$ such that $u\nsim_k u'$ for every $k>n$. 
\end{definition}
The purpose of this paper is to prove the following theorem.
\begin{theorem}\label{thetheorem}Let $\Gamma=\left(A,(\partition_n)_{n\geq 0},W\right)$ be an infinite game with a finite action set, a Borel winning set, perfect recall, and eventual perfect monitoring. Then $\Gamma$ is determined.\end{theorem}
The proof of the theorem relies on the stochastic extension of Martin's theorem about the determinacy of Blackwell games. However, except for the simple case in which the stages are divided into blocks and previous actions are monitored at the end of each block, I was unable to find an immediate reduction of the eventual perfect monitoring setup to the Blackwell games setup, nor was able to adapt Martin's proof to the eventual perfect monitoring setup. The difficulty stems from the fact that Martin's proof uses the existence of \emph{proper subgames} -- occurrences along the play path where the current partial history is commonly known. However, even under simple monitoring structures, such as when one player observes opponent's action with two periods delay, the game does not have proper subgames.

Infinite games with Borel winning sets have recently been used in economics literature on testing the quality of probabilistic predictions. Consider a forecaster who claims to know the probability distribution that governs some stochastic process. To prove his claim, the forecaster provides probabilistic predictions about the process. An inspector tests the forecaster's reliability using the infinite sequence of predictions provided by the forecaster and the observed realization of the process. Using Martin's Theorem about the determinacy of Blackwell games, I proved~\cite{shmaya-08} that any inspection which is based on predictions about the next-day realization of the process is manipulable, i.e., it can be strategically passed by a charlatan. Theorem~\ref{thetheorem} of this paper can be used to prove that inspections based on daily predictions about an arbitrarily long finite horizon are also manipulable~\cite[Section 5]{shmaya-08}. On the other hand, Olszewski and Sandroni~\cite{olszewski-sandroni-09} gave an example for a non-manipulable inspection which is based on prediction about the infinite future. Such an inspection relies on a game without eventual perfect monitoring.

In Section~\ref{sec-examples} I give some examples of games with and without eventual perfect monitoring. In Section~\ref{sec-compact} I prove the determinacy of infinite games with perfect recall and a compact winning set; this result is used in the proof of Theorem~\ref{thetheorem}. The proof of the theorem is in Section~\ref{sec-proof}. Section~\ref{sec-discussion} discusses the role of eventual perfect monitoring in the proof. Section~\ref{sec-open} discusses some open questions. 
Martin's Theorem is reviewed in the appendix.
\section{Endurance games}\label{sec-examples}
All the examples in this section have the same action set and the same winning set. The action set is $A=\{\texts,\textl\}$: At every stage, each player decides whether to Stay or Leave the game. Once a player leaves, his future actions do not affect the outcome of the game. For an infinite history $u=(a_0,a_1,\dots)$, let $n^1(u)=\min\left\{n\text{ even }\left|a_n=\textl\right.\right\}$ be the (possibly infinite) first stage in which player 1 left the game, and let $n^2(u)$ be the first stage in which player 2 left. Let \begin{equation}\label{def-w-examples}
W=\left\{u\in A^\bbn\left| \left(n^2(u)<n^1(u)<\infty\right)\text{ or }\left(n^1(u)<\infty\text{ and }n^2(u)=\infty\right)\right.\right\}
\end{equation}
be the winning set of player 1. The game models a `last man standing' endurance contest, as in a college drinking game or religious conflicts: Player 1 wants to leave the game at some point, but he doesn't want to be the first to leave. 

In Example~\ref{exm-1} both players have eventual perfect monitoring. In Example~\ref{exm-2} none of the players has eventual perfect monitoring. In Example~\ref{exm-3} only player 1 has eventual perfect monitoring.
\begin{example}\label{exm-1}Let $k$ be a positive integer. Assume that at stage $n$ each player observes his own actions and the actions of his opponent at stages smaller than $n-k$. Then the value of the game is $0$. An optimal strategy for player 2 is to play S as long as he is not informed that player 1 has played L. When player 2 knows that player 1 played L at some point, player 2 then plays L.\end{example}
Note that in the previous example, the number $k$ need not be constant. It can depend on the stage number, and can differ between the players. As long as player 2 knows the actions of his opponent eventually, the game is determined and the value is $0$. (The independence of the value on the information partitions in this example is not typical)
\begin{example}\label{exm-2}Assume that each player knows his own previous actions, but does not observe his opponent's actions. Then the game is not determined. In fact, $\underline\val~\Gamma=0$ and $\overline\val~\Gamma=1$.\end{example}
\begin{example}\label{exm-3}Assume that player 1 observes the past actions of player 2, but player 2 doesn't observe the past actions of player 1. Then the game is not determined. In fact, $\underline\val~\Gamma=1/2$ and $\overline\val~\Gamma=1$. An optimal strategy for player 1 is: At stage $0$ play $L$ or $S$ with probability $1/2$, and, at stage $2k$ for $k\geq 1$, play the action of player 2 from stage $2k-1$. \end{example}
\section{Games with a compact winning set}\label{sec-compact}
The set $A^\bbn$ of plays is naturally endowed with the product topology. In this section I prove the special case of Theorem~\ref{thetheorem} for compact winning sets. The determinacy follows from perfect recall alone. The proof uses two standard results from game theory: the Minimax Theorem for normal form games and Kuhn's Theorem.

Recall that a \emph{normal form game} is given by a triple $(\Sigma,\Theta,R)$ where $\Sigma$ and $\Theta$ are Borel spaces of \emph{pure strategies} for players 1 and 2, and $R:\Sigma\times\Theta\rightarrow [0,1]$ is a Borel \emph{payoff function}. A \emph{mixed strategy} $\xi$ of player 1 is a probability distribution over $\Sigma$. Mixed strategies $\tau$ of player 2 are defined analogously. The mixed extension of the normal form game $(\Sigma,\Theta,r)$ is \emph{determined} if
\[\sup_{\xi\in\Delta(\Sigma)}\inf_{\theta\in\Theta}\int R(\sigma,\theta)\xi(\textd \sigma)=\inf_{\tau\in\Delta(\Theta)}\sup_{\sigma\in\Sigma}\int R(\sigma,\theta)\tau(\textd \theta).\]
It follows from Fan's Minimax Theorem~\cite[Theorem 2]{fan-53} that if $\Sigma$ is a compact topological space and the function $R(\cdot,\theta)$ is upper semicontinuous for every $\theta\in\Theta$, then the mixed extension of the normal form game $(\Sigma,\Theta,R)$ is determined.

Let $\Gamma=(A, \{P_n\}_{n\in\bbn},W)$ be an infinite game with perfect recall. The \emph{normal form of $\Gamma$} is the normal form game $N(\Gamma)=(\Sigma,\Theta,R)$ defined as follows. A pure strategy $\sigma\in\Sigma$ of player 1 is a sequence $\{\sigma_n:P_n\rightarrow A\}_{n\text{ even }}$  of functions: at stage $n$, after the finite history $h=(h_0,h_1,\dots,h_{n-1})$ was played, player 1 plays $\sigma_n(\pi_n(h))$, where $\pi_n(h)$ is the atom of $P_n$ that contains $h$. Pure strategies $\theta$ of player 2 are defined analogously. Every pair $\sigma,\theta$ of pure strategies of players 1 and 2 determines an infinite history $u(\sigma,\theta)=(a_0,a_1,\dots)$ that is given by
\[
a_n=\begin{cases}\sigma_n\left(\pi_n\left(a_0,\dots,a_{n-1}\right)\right),&\text{ for even }n,\\\theta_n\left(\pi_n\left(a_0,\dots,a_{n-1}\right)\right),&\text{ for odd }n.\end{cases}\]The payoff function of $N(\Gamma)$ is $R(\sigma,\theta)=\bfone_W(u(\sigma,\theta))$. Kuhn's Theorem~\cite[Theorem D.1]{sorin} states the equivalence between mixed strategies and behavioral strategies in games with perfect recall. In particular, the game $\Gamma$ is determined if and only if its normal form game $N(\Gamma)$ is determined.
\begin{lemma}\label{lem-compact}An infinite game with a finite action set, perfect recall, and a compact winning set is determined.\end{lemma}
\begin{proof}
Let $\Gamma=(A, \{P_n\}_{n\in\bbn},W)$ be an infinite game with a finite action set $A$, perfect recall, and a compact winning set $W$.
By Kuhn's Theorem it is sufficient to prove that the normal form game $N(\Gamma)=(\Sigma,\Theta,R)$ of $\Gamma$ is determined. This follows from the minimax theorem. Indeed, the set $\Sigma$ of pure strategies of player 1 is compact in the product topology, and the payoff function $R$ is upper semicontinuous as a composition of the continuous function $(\sigma,\theta)\mapsto u(\sigma,\theta)$ and the function $\bfone_W$, which is upper semi continuous because $W$ is closed.
\end{proof}
\section{Proof of Theorem~\ref{thetheorem}}\label{sec-proof}
Throughout this section, we fixe a game $\Gamma$ that satisfies the conditions of Theorem 1.3.
\subsection*{Overview of the proof.}
Eventual perfect monitoring entails that the action of stage $m$ is known to the opponent at infinity. Lemma~\ref{lem-konig} below shows that in fact more is true: for every $m$ there exists some finite stage $n>m$ at which, regardless of the play path, the opponent knows the action of stage $m$. This means that at such a stage $n$ the action taken by the player at stage $m$ becomes \emph{common knowledge}: The opponent knows the action played at stage $m$, the player that played the action knows that the opponent knows, the opponent knows that the player that play the action knows that the opponent knows, ad infinitum.

Roughly speaking, I am going to construct a stochastic game $\Gamma^\ast$ with perfect information that mimics the original game $\Gamma$. In $\Gamma^\ast$, at every stage $m$, the player announces a mixture over the set $A$ of actions contingent on his information at that stage. So in $\Gamma^\ast$, instead of choosing an action which is not revealed to his opponent (as in $\Gamma$), the player announces how he intends to  randomize his action. The actual randomization is performed by Chance at a future stage $k(m)$, in which the action should become commonly known, and the realization of that randomization is immediately made public. So in the game $\Gamma^\ast$, Chance performs the randomization for the player. The determinacy of $\Gamma^\ast$ follows from Martin's Theorem, and I prove that the value of $\Gamma^\ast$ is also the value of the original game $\Gamma$. For this purpose I have to show that the fact that in $\Gamma^\ast$ the player announces his randomization plan cannot be used by the opponent to change the payoff in the game. This step, which is the core of the proof, uses approximations of the winning set by compact sets, and the fact that by Lemma~\ref{lem-compact} the original game $\Gamma$ is determined when the winning set is compact.


Since the sets of actions must be finite for Martin's Theorem to apply, I first prove that every behavioral strategy in $\Gamma$ can be approximated by a behavioral strategy in which all the mixtures are taken from some finite sets. This is done in Lemma~\ref{lem-coupling}. Because of the approximation argument, the stochastic game $\Gamma^\ast$ that is constructed in the proof depends on an additional parameter $\ep$ which corresponds to the level of approximation.
\subsection*{Preliminaries}
Let $A^{<\bbn}=\bigcup_{n\in\bbn}A^n$ be the set of finite histories of the game. For a finite history $h\in A^n$, the \emph{length of $h$} is given by $\length(h)=n$. For an infinite history $u=(a_0,a_1,a_2,\dots)\in A^\bbn$ and $n\in \bbn$, let $u|_n=(a_0,\dots,a_{n-1})\in A^{<\bbn}$ be the initial segment of $u$ of length $n$. Similarly, for a finite history $h\in A^{<\bbn}$ and $n<\length(h)$, let $h|_n$ be the initial segment of $h$ of length $n$.
\begin{lemma}\label{lem-konig}
For every $m\in\bbn$ there exists an $n>m$ such that $n\neq m\mod 2$ and such that at stage $n$  the opponent observes the action of stage $m$, i.e., for every pair $u=(a_0,a_1,\dots),u'=(a_0',a_1',\dots)$ of infinite histories $u\sim_n u'$ implies $a_m=a_m'$.\end{lemma}
\begin{proof}
Assume w.l.o.g. that $m$ is odd. Let $a\in A$, and let $C_a=\{u=(u_0,u_1,\dots)\in A^\bbn|u_m=a\}$. Then $C_a$ and $C_a^c$ are compact. Let $T_a\subseteq A^{<\bbn}$ be the set of all finite histories $h$ of even length $n$ such that $\pi_{n}^{-1}(h')\cap C_a\neq\emptyset$ and $\pi_{n}^{-1}(h')\cap C_a^c\neq\emptyset$, where $\pi_k(h)$ is the atom of $P_k$ that contains $h$. 

It follows from the perfect recall assumption that $T_a$ is a tree over $A^2$.
I claim that $T_a$ is well-founded. Indeed, if $v$ is an infinite branch of $T_a$, then $\bigcap_{n\geq m}\pi_n^{-1}(v|_n)\cap C_a$ and $\bigcap_{n\geq m}\pi_{n}^{-1}(v|_n)\cap C_a^c$ are nonempty as the intersections of decreasing sequences of compact sets. Let $u=(u_0,u_1,\dots)\in\bigcap_{n\geq m}\pi_{n}^{-1}(v|_n)\cap C_a$ and $u'=(u'_0,u'_1,\dots)\in\bigcap_{n\geq m}\pi_{n}^{-1}(v|_n)\cap C_a^c$. Then $u_m=a\neq u'_m$ and therefore $u\neq u'$, but $u\sim_{n} u'$ for every $n$, in contradiction to the eventual perfect monitoring assumption. 

By K\"{o}nig's Lemma, $T_a$ is finite. Let $n^a$ be the maximal length of elements of $T_a$, and let $n=\max\{n^a|a\in A\}+1 $. Then from stage $n$ onwards the opponent observes the action of stage $m$.
\end{proof}
\subsection*{Approximating strategies}
For two strategies $x,x'$ of player 1, let $d(x,x')$, the \emph{distance between $x$ and $x'$}, be given by
\[d(x,x')=\sum_{n\text{ even}}\max_{p\in P_n}\|x_n(p)-x_n'(p)\|_1,\]
where the maximum is taken over all atoms $p$ of $P_n$.
The distance $d(y,y')$ between two behavioral strategies $y,y'$ of player 2 is defined analogously.
\begin{lemma}\label{lem-coupling}Let $x,x'$ be strategies of player 1 and $y,y'$ be strategies of player 2. Then 
\[\|\mu_{x,y}(W)-\mu_{x',y'}(W)\|\leq\bigl(d(x,x')+d(y,y')\bigr)/2\]
for every Borel subset $W$ of $A^\bbn$.\end{lemma}
\begin{proof}The idea is to join a $(x,y)$-random play and a $(x',y')$-random play such that the two random plays are equal with high probability. Let $z_n:P_n\rightarrow A$ be given by $z_n=x_n$ for even $n$'s and $z_n=y_n$ for odd $n$'s and  $z'_n:P_n\rightarrow A$ be given by $z'_n=x'_n$ for even $n$'s and $z'_n=y'_n$ for odd $n$'s. Let 
$\alpha_0,\alpha_0',\alpha_1,\alpha_1',\dots$ be a sequence of $A$-valued random variables defined inductively such that
the conditional joint distribution of the pair $\left(\alpha_n,\alpha_n'\right)$ given the event $\{\alpha_i=a_i,\alpha_i'=a_i'\text{ for }0\leq i<n\}$ satisfies
\begin{align}
&\label{alpha-eq}\bbp\left(\alpha_n=a\left|\alpha_i=a_i,\alpha_i'=a_i'\text{ for }0\leq i<n\right.\right)=z_n(a_0,\dots,a_{n-1})[a],\\
&\label{alpha-prime-eq}\bbp\left(\alpha'_n=a'\left|\alpha_i=a_i,\alpha_i'=a_i'\text{ for }0\leq i<n\right.\right)=z'_n(a'_0,\dots,a'_{n-1})[a'],\text{ and}\\
&\label{diff-eq}\bbp\left(\alpha_n'\neq\alpha_n\left |\alpha_i=a_i,\alpha_i'=a_i'\text{ for }0\leq i<n\right.\right)\leq\|z_n(a_0,\dots,a_{n-1})-z_n'(a'_0,\dots,a'_{n-1})\|_1/2,
\end{align}
for every $n$ and every $a_0,a_0',\dots,a_{n-1},a_{n-1}'\in A$.
The existence of random variables $\alpha_n,\alpha_n'$ with the prescribed conditional distribution follows from a standard coupling argument~\cite[Theorem 5.2]{lindvall}.
From~(\ref{alpha-eq}) it follows that
\[\bbp\left(\alpha_n=a\left|\alpha_i=a_i\text{ for }0\leq i<n\right.\right)=z_n(a_0,\dots,a_{n-1})[a]\]
for every $n$ and every $a_0,\dots,a_{n-1}\in A$, i.e., that $\alpha_0,\alpha_1,\dots$ is an $(x,y)$-random play of $\Gamma(W)$. Similarly, from~(\ref{alpha-prime-eq}) it follows that $\alpha'_0,\alpha'_1,\dots$ is a $(x',y')$-random play of $\Gamma(W)$. From~(\ref{diff-eq}) it follows that 
\[\bbp\left(\alpha_n\neq\alpha'_n\left|\alpha_i=\alpha'_i\text{ for }0\leq i<n\right.\right)\leq \max_{p\in P_n}\|z_n(p)-z_n'(p)\|_1/2.\]
Therefore,
\begin{multline*}\bbp\left(\alpha_n\neq \alpha'_n\text{ for some }n\right)\leq\sum_{n\in\bbn}\bbp\left(\alpha_n\neq\alpha'_n\left|\alpha_i=\alpha'_i\text{ for }0\leq i<n\right.\right)\\\leq \sum_n\max_{p\in P_n}\|z_n(p)-z_n'(p)\|_1/2=\bigl(d(x,x')+d(y,y')\bigr)/2.\end{multline*}
The assertion follows from the last inequality and the fact that $\mu_{x,y}$ and $\mu_{x',y'}$ are the distributions of $\alpha_0,\alpha_1,\dots$ and $\alpha_0',\alpha_1',\dots$, respectively.
\end{proof}
\begin{corollary}\label{cor-coupling}
Let $\Delta_{\ep,n}$ be a finite set which is $\ep/{2^n}$-dense in $\Delta(A)$ endowed with $\|\|_1$, i.e., such that the $\ep/{2^n}$-balls around elements of $\Delta_{\ep,n}$ cover $\Delta(A)$. Then there exists an $\ep$-optimal strategy $y$ for player 2 in $\Gamma(W)$ such that $y_n(p)\in\Delta_{\ep,n}(A)$ for every odd $n$ and every atom $p$ of $P_n$.\end{corollary}
\begin{proof}
Let $y'$ be an $\ep/2$-optimal strategy of player 2 in $\Gamma(W)$ and let $y$ be a strategy of player 2 such that $\|y_n(p)-y_n'(p)\|_1<\ep/2^n$ and $y_n(p)\in\Delta_{\ep,n}(A)$ for every odd $n$ and every atom $p$ of $P_n$. Then $d(y,y')<\ep$, and therefore, 
\[\mu_{x,y}(W)\leq\mu_{x,y'}(W)+\ep/2\leq \overline\val\Gamma(W)+\ep\]for every strategy $x$ of player 1, where the first inequality follows from Lemma~\ref{lem-coupling}, and the second inequality from the fact that $y'$ is $\ep/2$-optimal. Therefore $y$ is $\ep$-optimal.\end{proof}
\subsection*{Chance as the players' randomization delegate}
Let $\Gamma=(A,P_n,W)$ be an infinite game with perfect recall and eventual perfect monitoring. In this section, I define an auxiliary stochastic game $\Gamma^\ast_\ep=\Gamma^\ast_\ep(W)$ with perfect information, which mimics the original game $\Gamma$. 

Fix $\ep > 0$ and, for every $n\in \bbn$, let $\Delta_{\ep,n}$ be a finite subset of the interior of $\Delta(A)$ which is $\ep/{2^n}$-dense in $\Delta(A)$ endowed with $\|\|_1$. For every $m\in\bbn$, fix $k(m)>m$ such that $m\neq k(m)~\mod 2$, and such that at stage $k(m)$ the opponent observes the action of stage $m$, as in Lemma~\ref{lem-konig}.

For every $n$, let $\Bpsi_n=\{\bpsi:P_n\rightarrow\Delta_{\ep,n}\}$ be the set of \emph{actions} of stage $n$ in $\Gamma_\ep^\ast(W)$, so that an action is a function from $P_n$ (viewed as a collection of atoms) to $\Delta_{\ep,n}$; and let $S_n=A^{K_n}$ be the set of \emph{states} of stage $n$ in $\Gamma^\ast_\ep(W)$, where $K_n=\{m|k(m)=n\}$. 
Let $f_n:A^n\rightarrow S_n$ be the projection over the corresponding coordinates $m\in K_n$, and let $F:S_0\times S_1\dots\rightarrow A^\bbn$ be such that 
\begin{equation}\label{F-f}F(f_0(u|_0),f_1(u|_1),\dots)=u\end{equation} for every $u\in A^\bbn$.

$\Gamma^\ast_\ep(W)$ is played as follows: Player 1 plays at even stages and player 2 at odd stages. At every stage $n$, Chance announces a state $s_n$ in $S_n$, and then the player that play at that stage announces an action $\bpsi_n$ in $\Bpsi_n$. Chance chooses the state $s_n$ of stage $n$ from the distribution $\zsigma\left(s_0,\bpsi_0,\dots,s_{n-1},\bpsi_{n-1}\right)$ that is given by
\begin{multline}\label{def-sigma}\zsigma\left(s_0,\bpsi_0,\dots,s_{n-1},\bpsi_{n-1}\right)[s]=\\
\bbp\left(f_n\left(\bar\alpha_0,\dots,\bar\alpha_{n-1}\right)=s\left|f_k(\bar\alpha_0,\dots\bar\alpha_{k-1})=s_k\text{ for }0\leq k<n-1\right.\right),\end{multline}
where $\bar\alpha_0,\dots,\bar\alpha_n$ is a sequence of $A$-valued random variables such that 
\begin{equation}\label{alpha-are-such-that}\bbp\left(\bar\alpha_k=a\left|\bar\alpha_0,\dots,\bar\alpha_{k-1}\right.\right)=\bpsi_k\left(\pi_k(\bar\alpha_0,\dots,\bar\alpha_{k-1}\right))[a].\end{equation}

A \emph{pure strategy} of player 1 in $\Gamma_\ep^\ast(W)$ is a sequence $\{x^\ast_n:S_0\times\Bpsi_0\times\dots\times S_{n-1}\times\Bpsi_{n-1}\times S_n\rightarrow\Bpsi_n\}_{n=0,2,\dots}$ of functions: at stage $n$, after observing the \emph{finite history} $\left(s_0,b_0,\dots,s_{n-1},b_{n-1},s_n\right)$, player 1 plays $x^\ast\left(s_0,b_0,\dots,s_{n-1},b_{n-1},s_n\right)$. Pure strategies $y^\ast$ of Player 2 are defined analogously. Let $X^\ast$ and $Y^\ast$ be the sets of pure strategies of players 1 and 2 respectively. The expected payoff for player 1 in the game $\Gamma_\ep^\ast(W)$ when the players play according to $(x^\ast,y^\ast)$ is given by
$R(x^\ast,y^\ast)=\bbp\left(F(\zeta_0,\zeta_1,\dots)\in W\right)$
where $\zeta_0,\beta_0,\zeta_1,\beta_1,\dots$ is a sequence of random variables, where the values of $\beta_n$ are in $B_n$ and the values of $\zeta_n$ are in $S_n$ such that 
\[\begin{split}&\bbp\left(\zeta_{n}=s\left|\zeta_0,\beta_{0},\dots,\zeta_{n-1},\beta_{n-1}\right.\right)=\zsigma\left(\zeta_0,\beta_{0},\dots,\zeta_{n-1},\beta_{n-1}\right)[s],\\&\beta_n=x^\ast_n\left(\zeta_0,\beta_{0},\dots,\zeta_{n-1},\beta_{n-1},\zeta_n\right)\text{ for even }n,\text{ and}\\&\beta_n=y^\ast_n\left(\zeta_0,\beta_{0},\dots,\zeta_{n-1},\beta_{n-1},\zeta_n\right)\text{ for odd }n.\end{split}\] 
I call such a sequence $\zeta_0,\beta_0,\zeta_1,\beta_1,\dots$ of random variables an \emph{$(x^\ast,y^\ast)$-random play} of $\Gamma^\ast_\ep(W)$.

Identifying the game $\Gamma_\ep^\ast(W)$ with its normal form, say that $\Gamma_\ep^\ast(W)$ is \emph{determined} if 
\[\sup_{\xi\in\Delta(X^\ast)}\inf_{y^\ast\in Y^\ast}\int R(x^\ast,y^\ast)\xi(\textd x^\ast)=\inf_{\tau\in\Delta(Y^\ast)}\sup_{x^\ast\in X^\ast}\int R(x^\ast,y^\ast)\tau(\textd y^\ast).\]
In this case the common value of the two sides of the last equations is called the \emph{value} of the game, and is denoted by $\val~\Gamma^\ast_\ep(W)$.
\begin{lemma}\label{cor-martin}
Let $W\subseteq A^\bbn$ be a Borel set. Then the game $\Gamma^\ast_\ep(W)$ is determined, and $\val~\Gamma^\ast_\ep(W_0)>\val~\Gamma^\ast_\ep(W)-\ep$ for some compact subset $W_0$ of $W$.\end{lemma}
\begin{proof}
In the terminology of appendix~\ref{sec-martin}, the game $\Gamma^\ast_\ep(W)$ is the stochastic game with stochastic setup $\cals=\left((S_n,\Bpsi_n)_{n\in\bbn},\zsigma\right)$ and the winning set $\eta^{-1}(W)$, where $\eta:S_0\times\Bpsi_0\times S_1\times\Bpsi_1\times\dots\rightarrow A^\bbn$ is the continuous map given by 
\begin{equation}\label{def-eta}
\eta\left(s_0,\bpsi_0,s_1,\bpsi_1,\dots\right)=F(s_0,s_1,\dots).\end{equation}
Thus $\eta^{-1}(W)$ is a Borel set and therefore by Proposition~\ref{pro-martin} 
the game $\left(\cals,\eta^{-1}(W)\right)$ is determined. Moreover, there exists a compact set 
$C\subseteq S_0\times\Bpsi_0\times S_1\times\Bpsi_1\times\dots$ such that $C\subseteq \eta^{-1}(W)$ and $\val(\cals,C)> \val(\cals,\eta^{-1}(W))-\ep$. Let $W_0=\eta(C)$. Then $W_0$ is a compact subset of $W$ and 
$\val(\cals,\eta^{-1}(W_0))\geq\val(\cals,C)>\val(\cals,\eta^{-1}(W))-\ep$, since $\eta^{-1}(W_0)\supseteq C$. The assertion follows from the fact that the games $\left(\cals,\eta^{-1}(W_0)\right)$ and $\left(\cals,\eta^{-1}(W)\right)$ are $\Gamma^\ast_\ep(W_0)$ and $\Gamma^\ast_\ep(W)$, respectively.\end{proof}
The following lemma says that, up to $\ep$, player 2 can guarantee in $\Gamma^\ast_\ep$ the same amount he can guarantee in $\Gamma$. 
\begin{lemma}\label{thelemma}For every Borel set $W$ of $A^\bbn$,
\[\val~\Gamma^\ast_\ep(W)-\ep\leq \overline\val~\Gamma(W).\]
\end{lemma}
\begin{proof}
Note first that by definition of $K_n$ and from the perfect recall assumption, there exist functions $g_{n,k}:P_n\rightarrow S_k$ for every $n$ and every $k\leq n$ such that 
\begin{equation}\label{def-g}
g_{n,k}(\pi_n(a_0,\dots,a_{n-1}))=f_k(a_0,\dots,a_{k-1})\end{equation}for every $h=(a_0,\dots,a_{n-1})\in A^n$ and where $\pi_n:A^n\rightarrow P_n$ is the natural projection.

Let $y$ be an $\ep$-optimal behavioral strategy for player 2 in $\Gamma(W)$ such that $y_n(p)\in\Delta_{\ep,n}(A)$ for every odd $n$ and every atom $p$ of $P_n$. The existence of such a strategy $y$ follows from Corollary~\ref{cor-coupling}.
Consider a pure strategy $y^\ast$ of player 2 in $\Gamma^\ast_\ep(W)$ that is given by $y^\ast_n(s_0,b_0,\dots,s_{n-1},b_{n-1},s_n)=y_n$ for every odd $n$ and every partial history $(s_0,b_0,\dots,s_{n-1},b_{n-1},s_n)$ of $\Gamma^\ast_\ep(W)$. (Thus, in every odd stage $n$, player 2's action is $y_n$, regardless of the history.) Let $x^\ast$ be any strategy of player 1 in $\Gamma^\ast_\ep$. Let $x$ be the behavioral strategy of player 1 in $\Gamma(W)$ that is given by 
\[x_n(p)=x^\ast_n(s_0,b_0,\dots,s_{n-1},b_{n-1},s_n)(p),\] where $\left(s_0,b_0,\dots,s_{n-1},b_{n-1},s_n\right)$ is the finite history of $\Gamma^\ast_\ep(W)$ defined inductively by $b_k=x^\ast_k\left(s_0,b_0,\dots,s_{k-1},b_{k-1},s_k\right)$ for even $k$, $b_k=y_k$ for odd $k$, and $s_k=g_{n,k}(p)$.  

I am going to join an $(x,y)$-random play of $\Gamma(W)$ and an $(x^\ast,y^\ast)$-random play of $\Gamma^\ast_\ep(W)$ with equal payoffs. Let $\Pi_0,\zeta_0,\beta_0,\alpha_0,\Pi_1,\zeta_1,\beta_1,\alpha_1,\dots$ be a sequence of random variables such that the values of $\zeta_n$ are in $S_n$, the values of $\beta_n$ are in $B_n$, and the values of $\alpha_n$ are in $A$, and such that
\begin{align}
&\label{def-zeta-play}\zeta_n=f_n\left(\alpha_0,\dots,\alpha_{n-1}\right),\\
&\label{def-even-play}\beta_n=x^\ast_n\left(\zeta_0,\beta_{0},\dots,\zeta_{n-1},\beta_{n-1},\zeta_n\right)\text{ for even }n,\\
&\label{def-odd-play}\beta_n=y^\ast_n\left(\zeta_0,\beta_{0},\dots,\zeta_{n-1},\beta_{n-1},\zeta_n\right)\text{ for odd }n\text{, and}\\
&\label{def-alpha-play}\bbp\left(\alpha_n=a\left |\alpha_0,\dots,\alpha_{n-1}\right.\right)=\beta_n\left(\Pi_{n}\right)[a].
\end{align}
From~(\ref{def-g}) and~(\ref{def-zeta-play}) it follows that 
\begin{equation}\label{zeta-g-delta-pi}
\zeta_k=g_{n,k}\left(\pi_n\left(\alpha_0,\dots,\alpha_{n-1}\right)\right)\end{equation}
for every $n$ and every $k\leq n$.
From~(\ref{def-odd-play}) and the definition of $y^\ast$, it follows that $\beta_n=y_n$ for every odd $n$. In particular,
\begin{equation}\label{follows-odd}y_n(\pi_n\left(\alpha_0,\dots,\alpha_{n-1}\right))=\beta_n(\pi_n\left(\alpha_0,\dots,\alpha_{n-1}\right))\end{equation}
for every odd $n$.
From~(\ref{def-even-play}), the definition of $x$,~(\ref{zeta-g-delta-pi}), and the fact that $\beta_k=y_k$ for every odd $k$, it follows that 
\begin{equation}\label{follows-even}x_n\left(\pi_n\left(\alpha_0,\dots,\alpha_{n-1}\right)\right)=\beta_n\left(\pi_n\left(\alpha_0,\dots,\alpha_{n-1}\right)\right)\end{equation}
for every even $n$. 
From~(\ref{def-alpha-play}), (\ref{follows-odd}), (\ref{follows-even}), it follows that 
\[\bbp\left(\alpha_n=a\left|\alpha_0,\dots,\alpha_{n-1}\right.\right)=\begin{cases}x_n(\pi_n(\alpha_0,\dots,\alpha_{n-1})),&n\text{ even,}\\y_n(\pi_n(\alpha_0,\dots,\alpha_{n-1}))&n\text{ odd,}\end{cases}\]
i.e., that $\alpha_0,\alpha_1,\dots$ is an $(x,y)$-random play of $\Gamma(W)$.

From~(\ref{def-zeta-play}),(\ref{def-even-play}),(\ref{def-odd-play}),(\ref{def-alpha-play}), and~(\ref{def-sigma}), it follows that
\begin{equation}\label{follows-zeta}\bbp(\zeta_n=s_n|\zeta_0,\beta_0,\dots,\zeta_{n-1},\beta_{n-1})=z\left(\zeta_0,\beta_0,\dots,\zeta_{n-1},\beta_{n-1}\right)[s_n].\end{equation}
Indeed, given the event $\{\zeta_0=s_0,\beta_0=b_0,\dots,\zeta_{n-1}=s_{n-1},\beta_{n-1}=b_{n-1}\}$, the conditional distribution of $\alpha_0,\dots,\alpha_{n-1}$ is like the conditional distribution of a sequence $\bar\alpha_0,\dots,\bar\alpha_{n-1}$ that satisfies~(\ref{alpha-are-such-that}) given that $f_k(\bar\alpha_0)=s_k$ for $k<n$. (Here I use the fact that $\beta_n$ is measurable with respect to $\zeta_0,\dots,\zeta_n$.)

From~(\ref{follows-zeta}),(\ref{def-even-play}), and~(\ref{def-odd-play}) it follows that $\zeta_0,\beta_0,\zeta_1,\beta_1,\dots$ is an $(x^\ast,y^\ast)$-random play of $\Gamma_\ep(W)$.
Therefore, the expected payoff for player 1 in $\Gamma^\ast_\ep(W)$ under $(x^\ast,y^\ast)$ is 
\[\bbp\left(F(\zeta_0,\zeta_1,\dots)\in W\right)=\bbp\left((\alpha_0,\alpha_1,\dots)\in W\right)=\mu_{x,y}(W)\leq \overline\val~\Gamma(W)+\ep,\]
where the first equality follows from~(\ref{F-f}) and (\ref{def-zeta-play}), the second equality from~(\ref{expected-via-alpha}), and the inequality from the fact that $y$ is $\ep$-optimal.

Summing up, I have provided a pure strategy $y^\ast$ of player 2 in $\Gamma^\ast_\ep(W)$ (namely, play $y_1,y_3,\dots$) that gives expected payoff of at most $\overline\val~\Gamma(W)+\ep$ against any pure strategy $x^\ast$ of player 1 in $\Gamma^\ast_\ep(W)$. Therefore, $\val~\Gamma^\ast_\ep(W)\leq \overline\val~\Gamma(W)+\ep$.
\end{proof}
\subsection*{Proof of Theorem~\ref{thetheorem}}Consider the stochastic game $\Gamma^\ast_\ep(W)$ defined above. Let $W_0$ be a compact subset of $W$ such that $\val~\Gamma^\ast_\ep(W_0)>\val~\Gamma^\ast_\ep(W)-\ep$, and whose existence follows from Lemma~\ref{cor-martin}.   
Then 
\[\underline\val~\Gamma(W)\geq \underline\val~\Gamma(W_0)=\overline\val~\Gamma(W_0)\geq \val~\Gamma^\ast_\ep(W_0)-\ep>\val~\Gamma^\ast_\ep(W)-2\ep,\]
where the first inequality follows from the fact that $W\supseteq W_0$, the first equality follows from Lemma~\ref{lem-compact}, the second inequality follows from Lemma~\ref{thelemma}, and the third inequality follows from the choice of $W_0$.

Similarly, for player 2 we get $\overline\val~\Gamma(W)< \val~\Gamma^\ast_\ep(W)+2\ep$. It follows that $\overline\val~\Gamma(W)< \underline\val~\Gamma(W)+4\ep$. Since $\ep$ was arbitrary, it follows that $\overline\val~\Gamma(W)=\underline\val~\Gamma(W)$.\qed
\section{The role of eventual perfect monitoring}\label{sec-discussion}
In the auxiliary game that I use in the proof of Theorem~\ref{thetheorem}, Chance plays the role of the players `randomization delegate', and all the actions of the players are perfectly monitored. What happens if we try to apply the same argument to a game without eventual perfect monitoring ? 
Consider the following game, which is a modification of the game in Example~\ref{exm-3} with Chance as player 1's randomization delegate.
\begin{example}\label{exm-3-prime}Let $A=\{\texts,\textl\}$ and consider the game with perfect monitoring that is played in stages $n=0,1,\dots$ as follows: Player 1 plays at even stages and player 2 plays at odd stages. At every even stage $n$ player 1 declares a mixed action $p_n\in [0,1]$, viewed as a dice over $A$ 
  ($p_n$ is the probability that player 1 will leave the game at stage $n$ if he didn't stop already). At every odd stage $n$ player 2 declares an action $A=\{\texts,\textl\}$. At infinity, Chance tosses the dice that player 1 declared one after another, thus obtaining an infinite history $h\in A^\infty$. Player 1 wins the game if $h\in W$, where $W$ is given by~(\ref{def-w-examples}). \end{example}
Unlike the original game of Example~\ref{exm-3}, the game of Example~\ref{exm-3-prime} is determined and has value $1/2$. An optimal strategy for player 2 is to leave the game at the first stage $n$ for which $(1-p_0)\cdot(1-p_2)\cdot\dots\cdot (1-p_{n-1})<1/2$. However, as we saw, in the game of Example~\ref{exm-3}, where player 1's action is not monitored by player 2, player 2 cannot guarantee to pay a payoff smaller than $1$.

So, the transition from imperfect monitoring to Chance acting as the player's randomization delegate might fundamentally change the strategic potential of game in the absence of eventual perfect monitoring. On the other hand, in games with eventual perfect monitoring, as with finite games, a player does not lose anything if he declares his mixed contingent plan and let Chance do the randomization for him.

The idea that a player declares his contingent mixed plan and let Chance randomize for him can be traced back to von Neumann and Morgenstern's discussion of mixed actions. They argue that a player who fears his plan will be found out or deduced by the opponent may want to use randomization. When the player let Chance choose his action, the opponent 
cannot possibly find out what the action is going to be, since the player does not 
know it himself. ``Ignorance is obviously a very good safeguard against disclosing information directly or indirectly''~\cite[Section 17.2.1]{von-neumann-morgenstern}
The fact that the game is determined shows that the player who follows the Founders' 
advice does not lose anything if his plan is indeed found out, so he might as well declare it. In the proof of Theorem~\ref{thetheorem} we walked the reverse path: Determinacy follows from the fact that a player does not lose anything if he declares his contingent plan. Example~\ref{exm-3-prime} shows that the argument relies on the assumption of eventual perfect monitoring.
\section{Open questions}\label{sec-open}
The following questions about possible extensions of Theorem~\ref{thetheorem} are interesting for their own sake, and also because they highlight the assumptions on the game that were used in the proof of Theorem~\ref{thetheorem}.

The first question is about relaxing eventual perfect monitoring. Say that an infinite game $\Gamma=\left(A,(\partition_n)_{n\geq 0},W\right)$ admits \emph{eventual outcome monitoring} if both players know the identity of the winner (that is, the outcome of the game) at infinity. Formally, this means that for every $u,u'\in A^\bbn$ such that $u\in W$ and $u'\notin W$, there exists an $n$ such that $u\nsim_k u'$ for every $k>n$. 
\begin{question*}Is every game with perfect recall and eventual outcome monitoring determined ?\end{question*}
The construction of the auxiliary game $\Gamma^\ast$ in the proof of Theorem~\ref{thetheorem} relies on Lemma~\ref{lem-konig}, that if both players know the play path at infinity then each action becomes commonly known at some finite time, and so the play path in $\Gamma$ can be reconstructed from Chance's actions in $\Gamma^\ast$. However, when both players know the occurrence of $W$ at infinity, it still needs not be the case that $W$ depends only on things that are commonly known at some finite times.

The second questions relates to general payoff functions. Assume that the game is played as in Section~\ref{sec-setup}, and at infinity player 2 pays player 1 the amount $f(a_0,a_1,\dots)$, where $f:A^\bbn\rightarrow [0,1]$ is a Borel function, the \emph{payoff function} of the game. The games studied in this paper are a special case, where $f=\bfone_W$ is an indicator function. The definition of determinacy for games with Borel payoff function is a straightforward extension of the definition in Section~\ref{sec-setup}.
\begin{question*}Is every game with eventual perfect monitoring and general payoff function determined ?\end{question*}
In the proof of Theorem~\ref{thetheorem} we approximated the winning set $W$ from below by a closed set, and used Lemma~\ref{lem-compact} that games with closed winning sets are determined. Lemma~\ref{lem-compact} and its proof hold for upper semi-continuous payoff functions. However, upper semi-continuous functions are not sufficient to approximate arbitrary payoff functions in Martin's Theorem about Blackwell Games. (one needs to use $G_\delta$-functions~\cite[p. 1576, Remark (c)]{martin-98}).

One can also combine the two questions: What about game with general payoff function, perfect recall and eventual payoff monitoring (i.e., for every histories $u,u'\in A^\bbn$ such that $f(u)\neq f(u')$ there exists an $n$ such that $u\nsim_k u'$ for every $k>n$) ? In this case the game needs not be determined. Indeed, Rosenberg and Vieille give an example of a non-determined recursive game with incomplete information on one side~\cite[Section 4.1, henceforth RV]{rosenberg-vieille}. To modify RV's game to an infinite game with eventual payoff monitoring, replace Nature's randomization in RV's game with two stages in which player 1 and then player 2 choose a bit from $\{0,1\}$ and the sum of these bits modulo $2$ is the state of nature. Player 1's initial choice is not monitored by player 2, but player 2's choice is monitored by player 1 so that player 1 knows the state of nature. Then the game proceeds as in RV's game until there is an absorption, in which case player 2 is notified of the initial bit chosen by player 1. If there is no absorption then at infinity player 2 will not know the initial bit of player 1, but will know that the payoff is $0$.
Since by choosing his initial bit uniformly every player can guarantee that the sum modulo $2$ is uniform, the indeterminacy of RV's game implies the indeterminacy of the modified game.
\appendix\section{Martin's Theorem for Stochastic Games}\label{sec-martin}
In this section I formulate Martin's Theorem about the determinacy of stochastic games. The stochastic game used in this paper has complete information, while Martin studied a more general setup in which the players play simultaneously. Note, however, that Martin's older theorem about the determinacy of Gale-Stewart games~\cite{martin-75} is not sufficient for my purposes because of the presence of Chance.

A \emph{stochastic game with perfect information} is given by $\left((S_n,\Bpsi_n)_{n\in\bbn},\zsigma, V\right)$, where $\Bpsi_0,B_1,\dots$ are finite sets of \emph{actions}, $S_0,S_1,\dots$ are finite sets of \emph{states} or \emph{Chance's actions}, $\zsigma=\{\zsigma_n:S_0\times\Bpsi_0\times\dots\times S_{n-1}\times\Bpsi_{n-1}\rightarrow \Delta(S_n)\}$ is \emph{Chance's strategy}, and $V\subseteq S_0\times\Bpsi_0\times S_1\times\Bpsi_1\times\dots$ is the \emph{winning set} of Player 1.

The game is played as follows: Player 1 plays at even stages and player 2 at odd stages. At every stage $n$, Chance announces a state $s_n$ in $S_n$, and then the player that plays at that stage announces an action $\bpsi_n$ in $\Bpsi_n$. Chance chooses the state $s_n$ of stage $n$ from the distribution $z(s_0,b_0,\dots,s_{n-1},b_{n-1})$. Player 1 wins the game if $\left(s_0,b_0,s_1,b_1,\dots\right)\in V$. 

I call a triple $\cals=\left((S_n,B_n)_{n\in\bbn},\zsigma,\right)$ of action sets, states sets, and Chance's strategy a \emph{stochastic setup}. So the stochastic games that I use in this paper are given by a stochastic setup $\cals=\left((S_n,\Bpsi_n)_{n\in\bbn},\zsigma\}\right)$ and a winning set $V\subseteq S_0\times\Bpsi_0\times S_1\times\Bpsi_1\times\dots$.

The definitions of strategies of the players, and of determinacy and value of the game, are omitted. 

The following proposition was proved by Martin~\cite{martin-98}. For the stochastic extension, see Maitra and Sudderth's paper~\cite{maitra-sudderth-98}. The fact that the lower value of the game can be approximated by the value on some compact subset was proved earlier by Maitra, Purves and Sudderth~\cite{maitra-purves-sudderth-92}.
\begin{proposition}\label{pro-martin}Let $\cals=\left((S_n,\Bpsi_n)_{n\in\bbn},\zsigma\}\right)$ be a stochastic setup, and let $V$ be a Borel subset of $S_0\times\Bpsi_0\times S_1\times\Bpsi_1\times\dots$. Then:
\begin{enumerate}
\item The game $\left(\cals,V\right)$ is determined.
\item For every $\ep>0$, there exists a compact subset $C$ of $V$ such that
\[\val\left(\cals,C\right)>\val\left(\cals,V\right)-\ep.\]\end{enumerate}
\end{proposition}
\bibliographystyle{plain}
\bibliography{references}

\begin{thebibliography}{10}

\bibitem{blackwell-69}
David Blackwell.
\newblock Infinite {$G_\delta$} games with imperfect information.
\newblock {\em Zastosowania Matematyki (Applicationes Mathematicae)},
  10:99--101, 1969.

\bibitem{blackwell-89}
David Blackwell.
\newblock Operator solution of infinite {$G_\delta$} games of imperfect
  information.
\newblock In Theodore~W. Anderson, Krishna~B. Athreya, and Donald~L. Igelhart,
  editors, {\em Probability, Statistics, and Mathematics}, pages 83--87.
  Academic Press, Boston, Massachusetts, 1989.

\bibitem{fan-53}
Ky~Fan.
\newblock Minimax theorems.
\newblock {\em Proceedings of the National Academy of Sciences}, 39:42--47,
  1953.

\bibitem{gale-stewart-53}
David Gale and F.M. Stewart.
\newblock Infinite games with perfect information.
\newblock {\em Annals of Mathematical Studies}, 28, 1953.

\bibitem{lindvall}
Torgny Lindvall.
\newblock {\em Lectures on the Coupling Method}.
\newblock Wiley, New York, 1992.

\bibitem{maitra-purves-sudderth-92}
Ashok~P. Maitra, Roger~A. Purves, and William~D. Sudderth.
\newblock Approximation theorems for gambling problems and stochastic games.
\newblock {\em Lecture Notes in Economics and Mathematical Systems},
  389:114--132, 1992.

\bibitem{maitra-sudderth-98}
Ashok~P. Maitra and William~D. Sudderth.
\newblock Finitely additive stochastic games with {B}orel measurable payoffs.
\newblock {\em International Journal of Game Theory}, 27(2):257--267, July
  1998.

\bibitem{martin-75}
Donald~A. Martin.
\newblock Borel determinacy.
\newblock {\em Annals of Mathematics}, 102(2):363--371, 1975.

\bibitem{martin-98}
Donald~A. Martin.
\newblock The determinacy of {B}lackwell games.
\newblock {\em Journal of Symbolic Logic}, 63(4):1565--1581, December 1998.

\bibitem{olszewski-sandroni-09}
Wojciech Olszewski and Alvaro Sandroni.
\newblock A nonmanipulable test.
\newblock {\em Ann. Statist.}, 37(2):1013--1039, 2009.

\bibitem{rosenberg-vieille}
Dinah Rosenberg, Eilon Solan, and Nicolas Vieille.
\newblock The {M}ax{M}in value of stochastic games with imperfect monitoring.
\newblock {\em Internat. J. Game Theory}, 32(1):133--150, 2003.
\newblock Special anniversary issue. Part 2.

\bibitem{scarf-shapley-57}
Herbert Scarf and L.S Shapley.
\newblock Games with partial information.
\newblock {\em Annals of Mathematics Study}, 28:213--229, 1578.

\bibitem{shmaya-08}
Eran Shmaya.
\newblock Many inspections are manipulable.
\newblock {\em Theoretical Economics}, 3(3):376--382, 2008.

\bibitem{sorin}
Sylvain Sorin.
\newblock {\em A First Course on Zero-Sum Repeated Games}.
\newblock Springer, Heidelberg, 2002.

\bibitem{vervoort-96}
Marco~R. Vervoort.
\newblock Blackwell games.
\newblock In Thomas~S. Ferguson, Lloyd~S. Shapley, and James~B. MacQueen,
  editors, {\em Statistics, Probability, and Game Theory}, pages 369--390.
  Institute of Mathematical Statistics, Hayward, California, 1996.

\bibitem{von-neumann-morgenstern}
John von Neumann and Oskar Morgenstern.
\newblock {\em Theory of games and economic behavior}.
\newblock Princeton University Press, Princeton, NJ, 2004.
\newblock Reprint of the 1980 edition [Princeton Univ. Press, Princeton, NJ;
  MR0565457].

\end{thebibliography}
\end{document}